\documentclass[11pt,a4paper]{article}

\usepackage[margin=2.54cm]{geometry}
\usepackage{array}
\newcolumntype{x}{>{\vbox to 5ex\bgroup\vfill\centering\arraybackslash\hspace{0pt}}m{2.2cm}<{\egroup}}
\usepackage[fleqn]{amsmath}
\usepackage{amssymb,latexsym}
\usepackage{amsthm}
\usepackage[section]{placeins}
\usepackage[colorlinks=true,linkcolor=black,citecolor=black,urlcolor=black]{hyperref}
\usepackage{enumitem}
\usepackage{graphicx}
\usepackage{tikz}

\newtheorem{theorem}{Theorem}[section]

\newtheorem{lemma}[theorem]{Lemma}
\newtheorem{proposition}[theorem]{Proposition}

\newcommand{\Z}{\mathbb{Z}}

\DeclareMathOperator{\Cay}{Cay}
\DeclareMathOperator{\Aut}{Aut}
\title{Clique-partitioned graphs}

\author{
Grahame Erskine\thanks{Open University, Milton Keynes, UK}\footnotemark[1]\\ \texttt{\small grahame.erskine@open.ac.uk}
\and Terry Griggs\footnotemark[1]\\ \texttt{\small terry.griggs@open.ac.uk}
\and Jozef \v{S}ir\'a\v{n}\thanks{Open University, Milton Keynes, UK and Slovak University of Technology, Bratislava}\footnotemark[2]\\ \texttt{\small jozef.siran@open.ac.uk}
}

\date{}

\begin{document}
\maketitle
\let\thefootnote\relax
\let\thefootnote\relax

\vspace*{-4ex}
\begin{abstract}
\noindent A graph $G$ of order $nv$ where $n\geq 2$ and $v\geq 2$ is said to be \emph{weakly $(n,v)$-clique-partitioned} if its vertex set can be decomposed in a unique way into $n$ vertex-disjoint $v$-cliques. It is \emph{strongly $(n,v)$-clique-partitioned} if in addition, the only $v$-cliques of $G$ are the $n$ cliques in the decomposition. We determine the structure of such graphs which have the largest possible number of edges.
\end{abstract}

\section{Introduction and motivation}\label{sec:intro}
In this paper we introduce and solve a problem in extremal graph theory. This concerns what we call \emph{clique-partitioned} graphs, and we consider two versions of the problem. The idea for this investigation came initially from a BBC TV quiz show ``Only Connect''. The premise of the quiz is that the teams are required to discover (often obscure) connections between clues. Our starting point is the penultimate round of the quiz, where each team is presented with a $4\times 4$ ``wall'' of clues which they are required to sort into 4 sets of 4. Each of the sets of clues is linked by a common attribute, but the difficulty is increased by the inclusion of ``red herring'' partial links between clues. There is only one way to decompose the 16 clues into 4 sets of 4, and the problem for the teams is to discover this unique solution by avoiding the distractions of the partial matches. We give a mathematically-themed example problem in Figure~\ref{fig:example}.
\begin{figure}
\large\centering
\begin{tabular}{|x|x|x|x|}
\hline
$\displaystyle\frac{1}{2}(1+i\sqrt{3})$ & $-i$ & $5$ & $e^{1+i\sqrt{2}}$ \\
\hline
$13$ & $\pi$ & $\displaystyle\frac{\log 2}{\log 3}$ & $\displaystyle -\frac{1}{2}$ \\
\hline
$2^{16}+1$ & $1$ & $\displaystyle\frac{22}{7}$ & $e^{i\pi/13}$ \\
\hline
$\displaystyle\frac{355}{113}$ & $3$ & $3.14$ & $2^{\sqrt{2}}$ \\
\hline
\end{tabular}
\caption{An example quiz}
\label{fig:example}
\end{figure}

At first sight there are a number of promising links. For example, $3.14$, $22/7$ and $355/113$ look like approximations to $\pi$. But is this a correct group? If so, should we include $3$ in it? What about $\pi$ itself? A few moments spent looking at the grid will reveal other potential matches. For example $\frac{1}{2}(1+i\sqrt{3})$, $e^{i\pi/13}$ and $e^{1+i\sqrt{2}}$ are complex numbers. But should we also include $-i$ which is purely imaginary? The numbers 1, 3, 5 and 13 all appear in the Fibonacci sequence. However the crucial feature of the grid is that there is one and only one way to decompose the 16 entries into 4 groups of 4, such that each group of 4 has some obvious link. The solution to this problem is shown at Figure~\ref{fig:solution}.

We may model this situation by means of a graph. The vertices of the graph will be the 16 clues, and two vertices will have an edge between them whenever there is a plausible link between the corresponding clues. There are two versions of the problem. In the first version, there are exactly four 4-cliques in the graph and the problem is simply to find them. In the second (harder) version, there may be other 4-cliques in the graph (corresponding to sets of four clues with a plausible link between them). However, there must be only one way to decompose the vertex set of the graph into four of these cliques.

However, our aim in this paper is not just to consider the specific case discussed above, but to deal with the more general problem of graphs which can be decomposed into a set of $n$ cliques each of size $v$, where $n\geq 2$ and $v\geq 2$ are integers. We formulate the problem in the language of graph theory.

\textbf{Definitions.}
Let $n\geq 2$ and $v\geq 2$. Let $G$ be a graph of order $nv$.

Then $G$ is \emph{weakly $(n,v)$-clique-partitioned} if its vertex set can be decomposed in a unique way into $n$ vertex-disjoint $v$-cliques.

$G$ is \emph{strongly $(n,v)$-clique-partitioned} if it is weakly clique-partitioned and in addition, the only $v$-cliques in $G$ are the $n$ cliques in the decomposition.

\textbf{Question 1.}
What is the largest possible number of edges in a strongly $(n,v)$-clique-partitioned graph?

\textbf{Question 2.}
What is the largest possible number of edges in a weakly $(n,v)$-clique-partitioned graph?

For both questions we prove an upper bound on the largest possible number of edges in the graph, and construct graphs which attain the upper bounds. We also determine their automorphism groups. In the case of weakly $(n,v)$-clique-partitioned graphs, these graphs are shown to be unique and are a generalisation of a classical result of Hetyei~\cite{Hetyei1964} on graphs with a unique perfect matching. Strongly $(n,v)$-clique-partitioned graphs with the largest possible number of edges are unique for the values $n=2$ and $v=2$ or $3$, and we give enumeration results for some other values of $(n,v)$. Other links to extremal graph theory are also discussed.

At this point it is appropriate to discuss the complementary problem. Clique partitions of a graph are equivalent to partitions into independent sets of the complementary graph, i.e. colourings. Define an \emph{equitable colouring} of a graph to be a colouring of the vertices of the graph such that adjacent vertices receive different colours, and the colour classes differ in cardinality by at most one. If the colour classes all have the same cardinality we have an \emph{exactly equitable colouring} or $(n,v)$-\emph{equitable colouring} where there are $n$ colour classes each of cardinality $v$. This corresponds precisely to the concept of weakly $(n,v)$-clique-partitioned in the complementary graph. Strongly $(n,v)$-clique-partitioned graphs would then correspond to $(n,v)$-equitable colourings in the complementary graph and with the additional property that there are no other independent sets of cardinality $v$ other than the colour classes. The reader is also referred to~\cite{deWerra} and~\cite{EllCaro}, although note that in both of these papers the terminology differs from that used here.

First, in Section~\ref{sec:strong} we deal with strongly clique-partitioned graphs, and then in Section~\ref{sec:weak} with weakly clique-partitioned graphs. In the main we present the proofs in the manner in which we have stated the problem, but of course they can also be done in terms of the complementary graph. However in Propositions~\ref{prop:circulant} and~\ref{prop:scpaut} it is more appropriate to work with the complementary problem.

\begin{figure}
	\Large\centering
	\begin{tabular}{|x|x|x|x|m{2.5cm}|}
		\hline
		$\displaystyle\frac{1}{2}(1+i\sqrt{3})$ & $-i$ & $e^{i\pi/13}$ & $1$ & \normalsize Roots of unity \\
		\hline
		$3.14$ & $\displaystyle\frac{355}{113}$ & $\displaystyle\frac{22}{7}$ & $\displaystyle -\frac{1}{2}$ & \normalsize Non-integer rationals \\
		\hline
		$\pi$ & $2^{\sqrt{2}}$ & $e^{1+i\sqrt{2}}$ & $\displaystyle\frac{\log 2}{\log 3}$ & \normalsize Transcendental numbers \\
		\hline
		$13$ & $3$ & $5$ & $2^{16}+1$ & \normalsize Primes \\
		\hline
	\end{tabular}
	
	\normalsize
	\caption{Solution to example quiz}
	\label{fig:solution}
\end{figure}
\section{Strongly clique-partitioned graphs}\label{sec:strong}
We begin with a lemma which gives an upper bound on the number of edges in a strongly clique-partitioned graph.

\begin{lemma}\label{lem:p1}
Let $n\geq 2$ and $v\geq 2$. Let $G$ be a strongly $(n,v)$-clique-partitioned graph. Then the number of edges in $G$ is at most
\[\frac{n}{2}v(v-1)+\frac{nv(n-1)(v-2)}{2}.\]
\end{lemma}
\begin{proof}
Each vertex $u$ of $G$ is adjacent to at most $v-2$ vertices in each clique not containing $u$, from which

$\quad \Delta(G)\leq (v-1)+(n-1)(v-2)$

and hence

$\quad |E(G)|\leq nv\left((v-1)+(n-1)(v-2)\right)/2$.
\end{proof}

We shall say that a strongly $(n,v)$-clique-partitioned graph attaining this bound is \emph{maximal}. It turns out that maximal graphs do in fact exist. To show this, we define the following graph $\Gamma(n,v)$ for any $n\geq 2$ and $v\geq 2$.

The vertex set of $\Gamma(n,v)$ is the set $\{(i,j):0\leq i\leq n-1,0\leq j\leq v-1\}$.
The adjacency rules are:
\[
(i,j)\sim(k,\ell) \iff
\begin{cases}
i=k &\text{ and } j\neq \ell\\
i<k &\text{ and } \ell-j\not\equiv 0\text{ or }1\pmod{v}\\
i>k &\text{ and } j-\ell\not\equiv 0\text{ or }1\pmod{v}\\
\end{cases}
\]

The construction is illustrated in Figure~\ref{fig:gamma} for the case $n=4,v=4$. In Figure~\ref{fig:gamma}(a) the $v$-cliques are formed by the first of the adjacency rules. The second and third rules are symmetric and define the edges between vertices in different numbered cliques; in Figure~\ref{fig:gamma}(b) we add the neighbours of vertex $(0,0)$ in clique 1 using these rules. In Figure~\ref{fig:gamma}(c) we add the neighbours in clique 1 of the remaining vertices in clique 0. In Figure~\ref{fig:gamma}(d) we add the neighbours of vertex $(0,0)$ in cliques 2 and 3, and in Figure~\ref{fig:gamma}(e) we do the same for the remaining vertices in clique 0. Finally in Figure~\ref{fig:gamma}(f) we complete the graph by adding the edges between cliques 1 and 2; cliques 1 and 3; and cliques 2 and 3.

\begin{figure}\centering
	\begin{tabular}{ccc}
		\begin{tikzpicture}[x=0.2mm,y=-0.2mm,inner sep=0.2mm,scale=0.75,thin,vertex/.style={circle,draw,minimum size=9,fill=white}]
\tiny
\node at (150,150) [vertex] (v1) {$00$};
\node at (210,150) [vertex] (v2) {$01$};
\node at (210,210) [vertex] (v3) {$02$};
\node at (150,210) [vertex] (v4) {$03$};
\node at (350,190) [vertex] (v5) {$10$};
\node at (410,190) [vertex] (v6) {$11$};
\node at (410,250) [vertex] (v7) {$12$};
\node at (350,250) [vertex] (v8) {$13$};
\node at (310,390) [vertex] (v9) {$20$};
\node at (370,390) [vertex] (v10) {$21$};
\node at (370,450) [vertex] (v11) {$22$};
\node at (310,450) [vertex] (v12) {$23$};
\node at (110,350) [vertex] (v13) {$30$};
\node at (170,350) [vertex] (v14) {$31$};
\node at (170,410) [vertex] (v15) {$32$};
\node at (110,410) [vertex] (v16) {$33$};
\path
	(v1) edge (v2)
	(v1) edge (v3)
	(v1) edge (v4)
	(v2) edge (v3)
	(v2) edge (v4)
	(v3) edge (v4)
	(v5) edge (v6)
	(v5) edge (v7)
	(v5) edge (v8)
	(v6) edge (v7)
	(v6) edge (v8)
	(v7) edge (v8)
	(v13) edge (v14)
	(v13) edge (v15)
	(v13) edge (v16)
	(v14) edge (v15)
	(v14) edge (v16)
	(v15) edge (v16)
	(v9) edge (v10)
	(v9) edge (v11)
	(v9) edge (v12)
	(v10) edge (v11)
	(v10) edge (v12)
	(v11) edge (v12)
	;
\end{tikzpicture}~~ & 
		\begin{tikzpicture}[x=0.2mm,y=-0.2mm,inner sep=0.2mm,scale=0.75,thin,vertex/.style={circle,draw,minimum size=9,fill=white}]
\tiny
\node at (150,150) [vertex] (v1) {$00$};
\node at (210,150) [vertex] (v2) {$01$};
\node at (210,210) [vertex] (v3) {$02$};
\node at (150,210) [vertex] (v4) {$03$};
\node at (350,190) [vertex] (v5) {$10$};
\node at (410,190) [vertex] (v6) {$11$};
\node at (410,250) [vertex] (v7) {$12$};
\node at (350,250) [vertex] (v8) {$13$};
\node at (310,390) [vertex] (v9) {$20$};
\node at (370,390) [vertex] (v10) {$21$};
\node at (370,450) [vertex] (v11) {$22$};
\node at (310,450) [vertex] (v12) {$23$};
\node at (110,350) [vertex] (v13) {$30$};
\node at (170,350) [vertex] (v14) {$31$};
\node at (170,410) [vertex] (v15) {$32$};
\node at (110,410) [vertex] (v16) {$33$};
\path
	(v1) edge (v7)
	(v1) edge (v8)
	(v1) edge (v2)
	(v1) edge (v3)
	(v1) edge (v4)
	(v2) edge (v3)
	(v2) edge (v4)
	(v3) edge (v4)
	(v5) edge (v6)
	(v5) edge (v7)
	(v5) edge (v8)
	(v6) edge (v7)
	(v6) edge (v8)
	(v7) edge (v8)
	(v13) edge (v14)
	(v13) edge (v15)
	(v13) edge (v16)
	(v14) edge (v15)
	(v14) edge (v16)
	(v15) edge (v16)
	(v9) edge (v10)
	(v9) edge (v11)
	(v9) edge (v12)
	(v10) edge (v11)
	(v10) edge (v12)
	(v11) edge (v12)
	;
\end{tikzpicture}~~ & 
		\begin{tikzpicture}[x=0.2mm,y=-0.2mm,inner sep=0.2mm,scale=0.75,thin,vertex/.style={circle,draw,minimum size=9,fill=white}]
\tiny
\node at (150,150) [vertex] (v1) {$00$};
\node at (210,150) [vertex] (v2) {$01$};
\node at (210,210) [vertex] (v3) {$02$};
\node at (150,210) [vertex] (v4) {$03$};
\node at (350,190) [vertex] (v5) {$10$};
\node at (410,190) [vertex] (v6) {$11$};
\node at (410,250) [vertex] (v7) {$12$};
\node at (350,250) [vertex] (v8) {$13$};
\node at (310,390) [vertex] (v9) {$20$};
\node at (370,390) [vertex] (v10) {$21$};
\node at (370,450) [vertex] (v11) {$22$};
\node at (310,450) [vertex] (v12) {$23$};
\node at (110,350) [vertex] (v13) {$30$};
\node at (170,350) [vertex] (v14) {$31$};
\node at (170,410) [vertex] (v15) {$32$};
\node at (110,410) [vertex] (v16) {$33$};
\path
	(v1) edge (v7)
	(v1) edge (v8)
	(v2) edge (v5)
	(v3) edge (v5)
	(v3) edge (v6)
	(v2) edge (v8)
	(v4) edge (v6)
	(v4) edge (v7)
	(v1) edge (v2)
	(v1) edge (v3)
	(v1) edge (v4)
	(v2) edge (v3)
	(v2) edge (v4)
	(v3) edge (v4)
	(v5) edge (v6)
	(v5) edge (v7)
	(v5) edge (v8)
	(v6) edge (v7)
	(v6) edge (v8)
	(v7) edge (v8)
	(v13) edge (v14)
	(v13) edge (v15)
	(v13) edge (v16)
	(v14) edge (v15)
	(v14) edge (v16)
	(v15) edge (v16)
	(v9) edge (v10)
	(v9) edge (v11)
	(v9) edge (v12)
	(v10) edge (v11)
	(v10) edge (v12)
	(v11) edge (v12)
	;
\end{tikzpicture}~~\\
		(a) & (b) & (c) \\
		\begin{tikzpicture}[x=0.2mm,y=-0.2mm,inner sep=0.2mm,scale=0.75,thin,vertex/.style={circle,draw,minimum size=9,fill=white}]
\tiny
\node at (150,150) [vertex] (v1) {$00$};
\node at (210,150) [vertex] (v2) {$01$};
\node at (210,210) [vertex] (v3) {$02$};
\node at (150,210) [vertex] (v4) {$03$};
\node at (350,190) [vertex] (v5) {$10$};
\node at (410,190) [vertex] (v6) {$11$};
\node at (410,250) [vertex] (v7) {$12$};
\node at (350,250) [vertex] (v8) {$13$};
\node at (310,390) [vertex] (v9) {$20$};
\node at (370,390) [vertex] (v10) {$21$};
\node at (370,450) [vertex] (v11) {$22$};
\node at (310,450) [vertex] (v12) {$23$};
\node at (110,350) [vertex] (v13) {$30$};
\node at (170,350) [vertex] (v14) {$31$};
\node at (170,410) [vertex] (v15) {$32$};
\node at (110,410) [vertex] (v16) {$33$};
\path
	(v1) edge (v7)
	(v1) edge (v8)
	(v1) edge (v11)
	(v1) edge (v12)
	(v1) edge (v15)
	(v1) edge (v16)
	(v2) edge (v5)
	(v3) edge (v5)
	(v3) edge (v6)
	(v2) edge (v8)
	(v4) edge (v6)
	(v4) edge (v7)
	(v1) edge (v2)
	(v1) edge (v3)
	(v1) edge (v4)
	(v2) edge (v3)
	(v2) edge (v4)
	(v3) edge (v4)
	(v5) edge (v6)
	(v5) edge (v7)
	(v5) edge (v8)
	(v6) edge (v7)
	(v6) edge (v8)
	(v7) edge (v8)
	(v13) edge (v14)
	(v13) edge (v15)
	(v13) edge (v16)
	(v14) edge (v15)
	(v14) edge (v16)
	(v15) edge (v16)
	(v9) edge (v10)
	(v9) edge (v11)
	(v9) edge (v12)
	(v10) edge (v11)
	(v10) edge (v12)
	(v11) edge (v12)
	;
\end{tikzpicture}~~ & 
		\begin{tikzpicture}[x=0.2mm,y=-0.2mm,inner sep=0.2mm,scale=0.75,thin,vertex/.style={circle,draw,minimum size=9,fill=white}]
\tiny
\node at (150,150) [vertex] (v1) {$00$};
\node at (210,150) [vertex] (v2) {$01$};
\node at (210,210) [vertex] (v3) {$02$};
\node at (150,210) [vertex] (v4) {$03$};
\node at (350,190) [vertex] (v5) {$10$};
\node at (410,190) [vertex] (v6) {$11$};
\node at (410,250) [vertex] (v7) {$12$};
\node at (350,250) [vertex] (v8) {$13$};
\node at (310,390) [vertex] (v9) {$20$};
\node at (370,390) [vertex] (v10) {$21$};
\node at (370,450) [vertex] (v11) {$22$};
\node at (310,450) [vertex] (v12) {$23$};
\node at (110,350) [vertex] (v13) {$30$};
\node at (170,350) [vertex] (v14) {$31$};
\node at (170,410) [vertex] (v15) {$32$};
\node at (110,410) [vertex] (v16) {$33$};
\path
	(v1) edge (v7)
	(v1) edge (v8)
	(v2) edge (v5)
	(v3) edge (v5)
	(v3) edge (v6)
	(v2) edge (v8)
	(v4) edge (v6)
	(v4) edge (v7)
	(v1) edge (v11)
	(v1) edge (v12)
	(v2) edge (v12)
	(v2) edge (v9)
	(v3) edge (v9)
	(v3) edge (v10)
	(v4) edge (v10)
	(v4) edge (v11)
	(v1) edge (v15)
	(v1) edge (v16)
	(v2) edge (v16)
	(v2) edge (v13)
	(v3) edge (v13)
	(v3) edge (v14)
	(v4) edge (v14)
	(v4) edge (v15)
	(v1) edge (v2)
	(v1) edge (v3)
	(v1) edge (v4)
	(v2) edge (v3)
	(v2) edge (v4)
	(v3) edge (v4)
	(v5) edge (v6)
	(v5) edge (v7)
	(v5) edge (v8)
	(v6) edge (v7)
	(v6) edge (v8)
	(v7) edge (v8)
	(v13) edge (v14)
	(v13) edge (v15)
	(v13) edge (v16)
	(v14) edge (v15)
	(v14) edge (v16)
	(v15) edge (v16)
	(v9) edge (v10)
	(v9) edge (v11)
	(v9) edge (v12)
	(v10) edge (v11)
	(v10) edge (v12)
	(v11) edge (v12)
	;
\end{tikzpicture}~~ & 
		\begin{tikzpicture}[x=0.2mm,y=-0.2mm,inner sep=0.2mm,scale=0.75,thin,vertex/.style={circle,draw,minimum size=9,fill=white}]
\tiny
\node at (150,150) [vertex] (v1) {$00$};
\node at (210,150) [vertex] (v2) {$01$};
\node at (210,210) [vertex] (v3) {$02$};
\node at (150,210) [vertex] (v4) {$03$};
\node at (350,190) [vertex] (v5) {$10$};
\node at (410,190) [vertex] (v6) {$11$};
\node at (410,250) [vertex] (v7) {$12$};
\node at (350,250) [vertex] (v8) {$13$};
\node at (310,390) [vertex] (v9) {$20$};
\node at (370,390) [vertex] (v10) {$21$};
\node at (370,450) [vertex] (v11) {$22$};
\node at (310,450) [vertex] (v12) {$23$};
\node at (110,350) [vertex] (v13) {$30$};
\node at (170,350) [vertex] (v14) {$31$};
\node at (170,410) [vertex] (v15) {$32$};
\node at (110,410) [vertex] (v16) {$33$};
\path
	(v1) edge (v7)
	(v1) edge (v8)
	(v2) edge (v5)
	(v3) edge (v5)
	(v3) edge (v6)
	(v2) edge (v8)
	(v4) edge (v6)
	(v4) edge (v7)
	(v5) edge (v11)
	(v5) edge (v12)
	(v6) edge (v12)
	(v6) edge (v9)
	(v7) edge (v9)
	(v7) edge (v10)
	(v8) edge (v10)
	(v8) edge (v11)
	(v1) edge (v11)
	(v1) edge (v12)
	(v2) edge (v12)
	(v2) edge (v9)
	(v3) edge (v9)
	(v3) edge (v10)
	(v4) edge (v10)
	(v4) edge (v11)
	(v1) edge (v15)
	(v1) edge (v16)
	(v2) edge (v16)
	(v2) edge (v13)
	(v3) edge (v13)
	(v3) edge (v14)
	(v4) edge (v14)
	(v4) edge (v15)
	(v5) edge (v15)
	(v5) edge (v16)
	(v6) edge (v16)
	(v6) edge (v13)
	(v7) edge (v13)
	(v7) edge (v14)
	(v8) edge (v14)
	(v8) edge (v15)
	(v9) edge (v15)
	(v9) edge (v16)
	(v10) edge (v16)
	(v10) edge (v13)
	(v11) edge (v13)
	(v11) edge (v14)
	(v12) edge (v14)
	(v12) edge (v15)
	(v1) edge (v2)
	(v1) edge (v3)
	(v1) edge (v4)
	(v2) edge (v3)
	(v2) edge (v4)
	(v3) edge (v4)
	(v5) edge (v6)
	(v5) edge (v7)
	(v5) edge (v8)
	(v6) edge (v7)
	(v6) edge (v8)
	(v7) edge (v8)
	(v13) edge (v14)
	(v13) edge (v15)
	(v13) edge (v16)
	(v14) edge (v15)
	(v14) edge (v16)
	(v15) edge (v16)
	(v9) edge (v10)
	(v9) edge (v11)
	(v9) edge (v12)
	(v10) edge (v11)
	(v10) edge (v12)
	(v11) edge (v12)
	;
\end{tikzpicture}~~\\
		(d) & (e) & (f) \\
	\end{tabular}
	\caption{Construction of the graph $\Gamma(4,4)$ by adding edges}
	\label{fig:gamma}
\end{figure}

Our aim is to show that the above graph is strongly $(n,v)$-clique-partitioned. In order to do this, we first remark that any $v$-clique in $\Gamma (n, v)$ contains precisely one vertex of the form $(i, j)$ for every $j \in \{0, 1,\ldots,v-1\}$.  Now we are in a position to prove the following theorem. 

\begin{theorem}\label{thm:p1}
Let $n\geq 2$ and $v\geq 2$. Then $\Gamma(n,v)$ has $nv\left((v-1)+(n-1)(v-2)\right)/2$ edges and is strongly $(n,v)$-clique-partitioned. Hence the upper bound of Lemma~\ref{lem:p1} is attained.
\end{theorem}

\begin{proof}
The number of edges $\Gamma(n,v)$ is immediate from the construction. Let $C$ be a $v$-clique in $\Gamma(n,v)$; then by the remark above, the vertices of $C$ are $\{(i_j,j):0\leq j\leq v-1\}$.

We proceed iteratively. Choose $x$ such that $i_x=\min\{i_j:0\leq j\leq v-1\}$. Now consider the vertices $(i_x,x)$ and $(i_{x+1},x+1)$. By the adjacency relations in $\Gamma(n,v)$ it follows that $i_x=i_{x+1}$.

Next consider the vertices $(i_{x+1},x+1)=(i_x,x+1)$ and $(i_{x+2},x+2)$. Again by the adjacency relations in $\Gamma(n,v)$ it follows that $i_x=i_{x+2}$.

The argument can now be repeated to get that all $i_j$, $j=0,1,\ldots,v-1$ are equal, which means that $C$ is one of the $v$-cliques into which the graph $\Gamma(n,v)$ can be strongly clique-partitioned.
\end{proof}

The algebraic construction of these graphs allows us to deduce some structural information about them. The graphs $\Gamma(n,v)$ are clearly regular, of order $nv$ and degree $(v-1)+(n-1)(v-2)=n(v-2)+1$, as are any strongly $(n,v)$-clique-partitioned graphs meeting the upper bound of Lemma~\ref{lem:p1}. In fact it turns out that the graphs $\Gamma(n,v)$ are Cayley graphs of a cyclic group, and hence vertex-transitive.

Recall that a \emph{Cayley graph} $\Cay(G,S)$ of a group $G$ and inverse-closed subset $S\subseteq G\setminus\{1\}$ has as vertex set the elements of $G$, and has edges from $g$ to $gs$ for every $g\in G$ and $s\in S$. A Cayley graph of a cyclic group is often called a \emph{circulant} graph.

\begin{proposition}\label{prop:circulant}
Let $n\geq 2$ and $v\geq 2$. Then $\Gamma(n,v)$ is a Cayley graph of the cyclic group of order $nv$. Its complement $\overline{\Gamma(n,v)}$ is isomorphic to $\Cay(\Z_{nv},\{\pm 1,\pm 2,\ldots,\pm(n-1)\})$.
\end{proposition}
\begin{proof}
Consider the map $\phi$ from the vertex set of $\Gamma(n,v)$ to $\Z_{nv}$ given by
\[\phi(i,j)=i-nj.\]
It is routine to show that $\phi$ is a bijection, and that the image set of the ``missing'' edges from vertex $(i,j)$ is the set $\{\phi(i,j)\pm 1,\ldots,\phi(i,j)\pm(n-1)\}$.
\end{proof}

Since $\Gamma(n,v)$ is a circulant graph of order $nv$, its automorphism group must contain a subgroup isomorphic to the dihedral group of order $2nv$. In fact our next result shows that if $v\geq 3$, this is the full automorphism group.

\begin{proposition}\label{prop:scpaut}
Let $n\geq 2$ and $v\geq 2$. Then $\Aut(\Gamma(n,v))$ is isomorphic to:
\begin{itemize}[label={~~~}]
	\item $S_2\wr S_n\cong \Z_2^n\rtimes S_n$, of order $2^n n!$, if $v=2$;
	\item the dihedral group $D_{2nv}$, of order $2nv$, if $v\geq 3$.
\end{itemize}
\end{proposition}

\begin{proof}
If $v=2$, then $\Gamma(n,v)$ consists of a set of $n$ vertex-disjoint edges, and the result follows immediately. 

So suppose $v\geq 3$ and let $G=\Cay(\Z_{nv},\{\pm 1,\pm 2,\ldots,\pm(n-1)\})$. Since $\overline{\Gamma(n,v)}\cong G$ we have $\Aut(\Gamma(n,v))\cong\Aut(G)$. Since $G$ is a circulant graph we know $D_{2nv}\leq\Aut(G)$ and we establish the reverse inclusion. Let $H$ be the graph with vertex set $V(H)=\Z_{nv}$ and where vertices $u,v$ are adjacent in $H$ if and only they are adjacent in $G$ and the intersection of their $G$-neighbourhoods has cardinality exactly $2n-4$. Then any automorphism of $G$ preserves $H$, so that $\Aut(G)\leq\Aut(H)$. It is straightforward to see that the edges of $H$ are precisely those edges of $G$ corresponding to the generators $\{\pm 1\}$, and so $H$ is the cyclic graph $C_{nv}$. Thus $\Aut(H)=D_{2nv}$ and the proof is complete.
\end{proof}

We now deal with some specific cases, beginning with $n=2$.

\begin{theorem}\label{thm:n2}
The graph $\Gamma(2,v)$ is the unique maximal strongly $(2,v)$-clique-partitioned graph, and is the complete graph $K_{2v}$ with a Hamiltonian cycle removed.
\end{theorem}

\begin{proof}
Let $G$ be a strongly $(2,v)$-clique-partitioned graph attaining the upper bound. Then by Lemma~\ref{lem:p1}, $G$ has $v(v-1)+v(v-2)$ edges and consists of two vertex-disjoint complete graphs $K_v$ joined by $v(v-2)$ edges. It is therefore the complete graph $K_{2v}$ from which $2v$ edges have been removed. Let the two $v$-cliques of $G$ be $X$ and $Y$. Since each vertex of $G$ has valency $(v-1)+(v-2)=2v-3$, the $2v$ edges which have been removed are a union of cycles, and all the edges in these cycles are between $X$ and $Y$. Suppose there is more than a single cycle, so that there is a cycle $x_1 y_1 x_2 y_2 \ldots x_m y_m x_1$ not containing all the vertices. Then the set $\{x_1,\ldots,x_m\}\cup (Y\setminus\{y_1,\ldots y_m\})$ is a further $v$-clique, contrary to the fact that $G$ is strongly clique-partitioned. Hence the edges removed must be a Hamiltonian cycle and so the graph must be $\Gamma(2, v)$.
\end{proof}

The case $v=2$ can also be solved completely.

\begin{theorem}\label{thm:v2}
The graph $\Gamma(n,2)$ is the unique maximal strongly $(n,2)$-clique-partitioned graph, and consists of $n$ vertex-disjoint edges.
\end{theorem}

\begin{proof}
From Lemma~\ref{lem:p1}, a graph attaining the bound has $n$ edges. Since the graph must contain $2n$ vertices, the result follows immediately.
\end{proof}

Next we deal with the case $v=3$. First we need the following definitions and lemma.

\textbf{Definitions.} A \emph{tournament} is a digraph in which every pair of distinct vertices is joined by precisely one arc. Equivalently, a tournament is a complete graph in which every edge has a specified orientation. A tournament is \emph{acyclic} if it contains no directed cycles.

First we observe that if a tournament is not acyclic then it must contain a directed 3-cycle. Suppose that a tournament contains a directed $m$-cycle, $(a_1,a_2,\ldots,a_m)$ where $m\geq 4$. If the tournament contains the arc $(a_3,a_1)$ then it contains a directed 3-cycle. Otherwise it contains the arc $(a_1,a_3)$ and so contains a directed $(m-1)$-cycle, $(a_1,a_3,\ldots,a_m)$. Repeating this argument, we see that the tournament contains a directed 3-cycle.

The following lemma is well-known, see for example~\cite[Theorem 7.13]{Foulds2012}.

\begin{lemma}\label{lem:tournament}
For every $n\geq 2$, there exists an acyclic tournament on $n$ vertices which is unique up to isomorphism.
\end{lemma}

We are now in a position to prove the next theorem.

\begin{theorem}\label{thm:v3}
The graph $\Gamma(n,3)$ is the unique maximal strongly $(n,3)$-clique-partitioned graph.
\end{theorem}

\begin{proof}
Let $G$ be a strongly $(n,3)$-clique-partitioned graph attaining the bound. We take the vertex set of $G$ to be the set $\{(i,j):1\leq i\leq n-1,0\leq j\leq 2\}$ where the cliques are defined by the adjacencies $(i,j)\sim (i,k)$ for all $i=0,1,\ldots,n-1$ and $j,k=0,1,2,j\neq k$. From Lemma~\ref{lem:p1}, each vertex is also joined by one edge to each of the $n-1$ cliques not containing that vertex. So, without loss of generality, we can assume additional adjacencies $(0,j)\sim (i,j+2)$ for all $i=1,2,\ldots,n-1$ and $j=0,1,2$.

When $n=2$, there are no further adjacencies and the graph is $\Gamma(2,3)$, the triangular prism. (See Figure~\ref{fig:gamma23}.)

\begin{figure}\centering
	\begin{tikzpicture}[x=0.2mm,y=-0.2mm,inner sep=0.2mm,scale=0.6,thin,vertex/.style={circle,draw,minimum size=10,fill=white}]
\tiny
\node at (130.1,400) [vertex] (v1) {$02$};
\node at (221,242.5) [vertex] (v2) {$00$};
\node at (311.9,400) [vertex] (v3) {$01$};
\node at (468.1,485) [vertex] (v4) {$11$};
\node at (559,327.5) [vertex] (v5) {$12$};
\node at (649.9,485) [vertex] (v6) {$10$};
\path
	(v1) edge (v2)
	(v1) edge (v3)
	(v2) edge (v3)
	(v4) edge (v5)
	(v4) edge (v6)
	(v5) edge (v6)
	(v2) edge (v5)
	(v3) edge (v6)
	(v1) edge (v4)
	;
\end{tikzpicture}
	\caption{The graph $\Gamma(2,3)$}
	\label{fig:gamma23}
\end{figure}

When $n=3$, there are further adjacencies between clique 1 and clique 2. If $(1,j)\sim(2,j)$, $j=0,1,2$, further cliques would be introduced, so there are two possibilities:
\begin{enumerate}[label=(\alph*),topsep=-1ex,itemsep=-1ex]
	\item $(1,j)\sim(2,j+2), j=0,1,2$,\newline
		or
	\item $(1,j)\sim(2,j+1), j=0,1,2$.
\end{enumerate}
However, the graphs obtained by the two possibilities are isomorphic by the permutation
\[\big((1,0)\,(2,0)\big)\big((1,1)\,(2,1)\big)\big((1,2)\,(2,2)\big).\]
(See Figure~\ref{fig:gamma33}.)

\begin{figure}\centering
	\begin{tabular}{cc}
		\begin{tikzpicture}[x=0.2mm,y=-0.2mm,inner sep=0.2mm,scale=0.6,thin,vertex/.style={circle,draw,minimum size=10,fill=white}]
\tiny
\node at (250.4,606.5) [vertex] (v1) {$21$};
\node at (341.3,449) [vertex] (v2) {$22$};
\node at (432.3,606.5) [vertex] (v3) {$20$};
\node at (154.4,318) [vertex] (v4) {$02$};
\node at (245.3,160.5) [vertex] (v5) {$00$};
\node at (336.2,318) [vertex] (v6) {$01$};
\node at (492.4,403) [vertex] (v7) {$11$};
\node at (583.3,245.5) [vertex] (v8) {$12$};
\node at (674.2,403) [vertex] (v9) {$10$};
\path
	(v1) edge (v2)
	(v1) edge (v3)
	(v2) edge (v3)
	(v4) edge (v5)
	(v4) edge (v6)
	(v5) edge (v6)
	(v7) edge (v8)
	(v7) edge (v9)
	(v8) edge (v9)
	(v5) edge (v8)
	(v6) edge (v9)
	(v4) edge (v7)
	(v2) edge (v5)
	(v3) edge (v6)
	(v1) edge (v4)
	(v1) edge (v8)
	(v3) edge (v7)
	(v2) edge (v9)
	;
\end{tikzpicture} & \begin{tikzpicture}[x=0.2mm,y=-0.2mm,inner sep=0.2mm,scale=0.6,thin,vertex/.style={circle,draw,minimum size=10,fill=white}]
\tiny
\node at (250.4,606.5) [vertex] (v1) {$21$};
\node at (341.3,449) [vertex] (v2) {$22$};
\node at (432.3,606.5) [vertex] (v3) {$20$};
\node at (154.4,318) [vertex] (v4) {$02$};
\node at (245.3,160.5) [vertex] (v5) {$00$};
\node at (336.2,318) [vertex] (v6) {$01$};
\node at (492.4,403) [vertex] (v7) {$11$};
\node at (583.3,245.5) [vertex] (v8) {$12$};
\node at (674.2,403) [vertex] (v9) {$10$};
\path
	(v1) edge (v2)
	(v1) edge (v3)
	(v2) edge (v3)
	(v4) edge (v5)
	(v4) edge (v6)
	(v5) edge (v6)
	(v7) edge (v8)
	(v7) edge (v9)
	(v8) edge (v9)
	(v5) edge (v8)
	(v6) edge (v9)
	(v4) edge (v7)
	(v2) edge (v5)
	(v3) edge (v6)
	(v1) edge (v4)
	(v3) edge (v8)
	(v1) edge (v9)
	(v2) edge (v7)
	;
\end{tikzpicture} \\
		(a) & (b)
	\end{tabular}
	\caption{Constructing strongly $(3,3)$-clique-partitioned graphs}
	\label{fig:gamma33}
\end{figure}

Therefore, again without loss of generality, we may assume possibility (a). The graph obtained is $\Gamma(3,3)$ and is the graph Q24 illustrated in~\cite[p.145]{Read2005}.

When $n=4$, there are further adjacencies between clique $i$ and clique $k$, $1\leq i<k\leq 3$, i.e. $(i,k)\in\{(1,2),(1,3),(2,3)\}$. Again there are two possibilities:
\begin{enumerate}[label=(\alph*),topsep=-1ex,itemsep=-1ex]
	\item $(i,j)\sim(k,j+2), j=0,1,2$,\newline
	or
	\item $(i,j)\sim(k,j+1), j=0,1,2$.
\end{enumerate}
Note that possibility (b) can be written as $(k,j)\sim(i,j+2), j=0,1,2$. We denote possibility~(a) by $i\to k$ and possibility~(b) by $k\to i$. So up to isomorphism, there are two cases to consider:
\begin{enumerate}[label=(\roman*),topsep=-1ex,itemsep=-1ex]
	\item $1\to 2,\quad 2\to 3,\quad 1\to 3$,\newline
	and
	\item $1\to 2,\quad 2\to 3,\quad 3\to 1\quad$ (which is a directed 3-cycle).
\end{enumerate}
Case (i) gives the graph $\Gamma(4,3)$ but case (ii) introduces further cliques, for example the vertices $(1,0)$, $(2,2)$ and $(3,1)$.

Now let $n\geq 5$. There are further adjacencies between clique $i$ and clique $k$, $1\leq i<k\leq n-1$ and, following the same procedure as for the case $n=4$, these can be determined by a tournament on the vertex set $\{i:1\leq i\leq n-1\}$. However, in order not to introduce further cliques, the tournament must not contain a directed 3-cycle, i.e. it must be acyclic.

Thus from Lemma~\ref{lem:tournament}, a strongly $(n,3)$-clique-partitioned graph ($n\geq 5$) attaining the upper bound of Lemma~\ref{lem:p1} is unique and is the graph $\Gamma(n,3)$.
\end{proof}

The next case to consider is $n=3$ and $v=4$. We have the following result.

\begin{theorem}\label{thm:n3v4}
There are precisely two maximal strongly $(3,4)$-clique-partitioned graphs.
\end{theorem}
\begin{proof}
Let the vertex set of a maximal strongly $(3,4)$-clique-partitioned graph be the set $\{(i,j):0\leq i\leq 2,0\leq j\leq 3\}$ where the cliques are defined by the adjacencies $(i,j)\sim(i,\ell)$, $0\leq i\leq 2$, $0\leq j<\ell\leq 3$. In order to simplify the notation we will now denote vertices $(0,j)$, $(1,j)$ and $(2,j)$ by $Aj$, $Bj$ and $Cj$ respectively. Ignoring the vertices $Cj$ and the edges incident with these vertices, the reduced graph is strongly $(2,4)$-clique-partitioned and is therefore the graph $\Gamma(2,4)$. 

From Theorem~\ref{thm:n2}, the missing edges between any two cliques form an 8-cycle, and therefore the edges joining the vertices $Aj$ and $Bj$ also form an 8-cycle. Without loss of generality, we can assume it to be $(A0,B2,A3,B1,A2,B0,A1,B3)$. Then, also without loss of generality, the edges joining the vertices $Aj$ and $Cj$ form one of the following 8-cycles.
\begin{enumerate}[label=(\Roman*),itemsep=0pt]
	\item $(A0,C2,A3,C1,A2,C0,A1,C3)$,
	\item $(A0,C2,A2,C1,A1,C0,A3,C3)$,
	\item $(A0,C2,A2,C1,A3,C0,A1,C3)$.
\end{enumerate}
However, cases (II) and (III) are isomorphic under the permutation $(A1\,A3) (B0\,B1) (B2\,B3)$. Thus there are only the two cases (I) and (II) to consider. It remains to determine the edges joining the vertices $Bj$ and $Cj$, which again must form an 8-cycle.

Consider first case (I). It is not possible that $Bj\sim Cj$ for any $j$, since a 4-clique on the vertex set $\{A(j+1),A(j+2),Bj,Cj\}$ would be created. This leaves just six possibilities.
\begin{enumerate}[label=(\roman*),itemsep=0pt]
	\item $(B0,C1,B2,C3,B1,C0,B3,C2)$,
	\item $(B0,C1,B3,C2,B1,C0,B2,C3)$,
	\item $(B0,C1,B3,C0,B2,C3,B1,C2)$,
	\item $(B0,C2,B3,C1,B2,C0,B1,C3)$,
	\item $(B0,C2,B1,C0,B3,C1,B2,C3)$,
	\item $(B0,C1,B2,C0,B3,C2,B1,C3)$.
\end{enumerate}

However, the graphs obtained from possibilities (i) and (ii), (iii) and (iv), (v) and (vi) are isomorphic under the permutation $(B0\,C0)(B1\,C1)(B2\,C2)(B3\,C3)$. Further, possibilities (i) and (vi) are isomorphic under the permutation $(A0\,A1\,A2\,A3)(B0\,B1\,B2\,B3)(C0\,C1\,C2\,C3)$. This leaves just two possibilities, (i) and (iv), the latter of which gives rise to the graph $\Gamma(3,4)$.

Denote the graph obtained from the former possibility by $\Theta(3,4)$. It remains to prove that this is indeed a maximal strongly $(3,4)$-clique-partitioned graph and is not isomorphic to $\Gamma(3,4)$. Both these tasks are readily accomplished using computer packages~\cite{GAP4,GRAPE} which also allow us to compute the automorphism groups of the two graphs. For $\Theta(3,4)$, the automorphism group is a Klein 4-group generated by the permutations $(A1\,A3)(B0\,B1)(B2\,B3)(C0\,C1)(C2\,C3)$ and $(A0\,A2)(B0\,C3)(B1\,C2)(B2\,C1)(B3\,C0)$; and for $\Gamma(3,4)$ it is the dihedral group of order 24, in agreement with Proposition~\ref{prop:scpaut}.

Next consider case (II). It is not possible that $B2\sim C3$ since a 4-clique $\{A0,A3,B2,C3\}$ would be created; nor that $B0\sim C1$ since a 4-clique $\{A1,A2,B0,C1\}$ would be created. This reduces the number of possibilities for the 8-cycle on the vertices $Bj$ and $Cj$ to 20; 16 of which can also be eliminated because further 4-cliques are formed. In Table~\ref{tab:n3v4cycles} we give a list of these possibilities, together with any cliques which occur.

We find that the graphs obtained from case (II), possibilities (i), (vii), (xii) and (xviii) are isomorphic to the graph obtained from case (I), possibility (i) by the following permutations.

\begin{tabular}{rl}
	Possibility & Permutation \\
	(i) & $(A0\,B1)(A1\,B0)(A2\,B3)(A3\,B2)(C0\,C2\,C1)$ \\
	(vii) & $(A0\,B1\,C1)(A1\,B0\,C3\,A2\,B3\,C0)(A3\,B2\,C2)$ \\
	(xii) & $(A0\,C0\,B3\,A1\,C1\,B1)(A2\,C2\,B0)(A3\,C3\,B2)$ \\
	(xviii) & $(A0\,C0\,A2\,C2\,A3\,C3\,A1\,C1)(B0\,B2\,B3\,B1)$ \\
\end{tabular}

\end{proof}
It may be useful to describe the relationship between the two graphs $\Gamma(3,4)$ and $\Theta(3,4)$. The adjacencies between the cliques in the former graph are given by the Hamiltonian cycles $(X0,Y2,X3,Y1,X2,Y0,X1,Y3)$ where $(X,Y)\in\{(A,B),(A,C),(B,C)\}$. For the graph $\Theta(3,4)$ the edges $(B0,C3)$, $(B2,C0)$ and $(B3,C1)$ are replaced by $(B0,C1)$, $(B2,C3)$ and $(B3,C0)$. The two graphs are illustrated in Figure~\ref{fig:n3v4}.

\begin{figure}[h]\centering
\begin{tabular}{cc}
\begin{tikzpicture}[x=0.2mm,y=-0.2mm,inner sep=0.2mm,scale=0.6,thick,vertex/.style={circle,draw,minimum size=12,fill=lightgray}]
\tiny
\node at (370,145) [vertex,fill=white] (v1) {$A0$};
\node at (258,175) [vertex,fill=white] (v2) {$B0$};
\node at (175,258) [vertex,fill=white] (v3) {$C0$};
\node at (145,370) [vertex,fill=white] (v4) {$A3$};
\node at (175,483) [vertex,fill=white] (v5) {$B3$};
\node at (258,565) [vertex,fill=white] (v6) {$C3$};
\node at (370,595) [vertex,fill=white] (v7) {$A2$};
\node at (483,565) [vertex,fill=white] (v8) {$B2$};
\node at (565,483) [vertex,fill=white] (v9) {$C2$};
\node at (595,370) [vertex,fill=white] (v10) {$A1$};
\node at (565,258) [vertex,fill=white] (v11) {$B1$};
\node at (483,175) [vertex,fill=white] (v12) {$C1$};
\path
	(v1) edge (v4)
	(v1) edge (v5)
	(v1) edge (v6)
	(v1) edge (v7)
	(v1) edge (v8)
	(v1) edge (v9)
	(v1) edge (v10)
	(v2) edge (v5)
	(v2) edge (v6)
	(v2) edge (v7)
	(v2) edge (v8)
	(v2) edge (v9)
	(v2) edge (v10)
	(v2) edge (v11)
	(v3) edge (v6)
	(v3) edge (v7)
	(v3) edge (v8)
	(v3) edge (v9)
	(v3) edge (v10)
	(v3) edge (v11)
	(v3) edge (v12)
	(v4) edge (v7)
	(v4) edge (v8)
	(v4) edge (v9)
	(v4) edge (v10)
	(v4) edge (v11)
	(v4) edge (v12)
	(v5) edge (v8)
	(v5) edge (v9)
	(v5) edge (v10)
	(v5) edge (v11)
	(v5) edge (v12)
	(v6) edge (v9)
	(v6) edge (v10)
	(v6) edge (v11)
	(v6) edge (v12)
	(v7) edge (v10)
	(v7) edge (v11)
	(v7) edge (v12)
	(v8) edge (v11)
	(v8) edge (v12)
	(v9) edge (v12)
	;
\end{tikzpicture} &
\begin{tikzpicture}[x=0.2mm,y=-0.2mm,inner sep=0.2mm,scale=0.6,thick,vertex/.style={circle,draw,minimum size=12,fill=lightgray}]
\tiny
\node at (370,145) [vertex,fill=white] (v1) {$A0$};
\node at (258,175) [vertex,fill=white] (v2) {$B0$};
\node at (175,258) [vertex,fill=white] (v3) {$C0$};
\node at (145,370) [vertex,fill=white] (v4) {$A3$};
\node at (175,483) [vertex,fill=white] (v5) {$B3$};
\node at (258,565) [vertex,fill=white] (v6) {$C3$};
\node at (370,595) [vertex,fill=white] (v7) {$A2$};
\node at (483,565) [vertex,fill=white] (v8) {$B2$};
\node at (565,483) [vertex,fill=white] (v9) {$C2$};
\node at (595,370) [vertex,fill=white] (v10) {$A1$};
\node at (565,258) [vertex,fill=white] (v11) {$B1$};
\node at (483,175) [vertex,fill=white] (v12) {$C1$};
\path
	(v1) edge (v4)
	(v1) edge (v5)
	(v1) edge (v6)
	(v1) edge (v7)
	(v1) edge (v8)
	(v1) edge (v9)
	(v1) edge (v10)
	(v2) edge (v5)
	(v2) edge (v7)
	(v2) edge (v8)
	(v2) edge (v9)
	(v2) edge (v10)
	(v2) edge (v11)
	(v3) edge (v6)
	(v3) edge (v7)
	(v3) edge (v9)
	(v3) edge (v10)
	(v3) edge (v11)
	(v3) edge (v12)
	(v4) edge (v7)
	(v4) edge (v8)
	(v4) edge (v9)
	(v4) edge (v10)
	(v4) edge (v11)
	(v4) edge (v12)
	(v5) edge (v8)
	(v5) edge (v9)
	(v5) edge (v10)
	(v5) edge (v11)
	(v6) edge (v9)
	(v6) edge (v10)
	(v6) edge (v11)
	(v6) edge (v12)
	(v7) edge (v10)
	(v7) edge (v11)
	(v7) edge (v12)
	(v8) edge (v11)
	(v8) edge (v12)
	(v9) edge (v12)
	(v2) edge (v12)
	(v6) edge (v8)
	(v3) edge (v5)
	;
\end{tikzpicture} \\
$\Gamma(3,4)$ & $\Theta(3,4)$
\end{tabular}
\caption{The two strongly $(3,4)$-clique-partitioned graphs}
\label{fig:n3v4}
\end{figure}

\begin{table}[h]\centering
	\begin{tabular}{rll}
		& Cycle & Clique formed \\
		(i) & $(B0,C3,B1,C1,B2,C0,B3,C2)$ &  \\
		(ii) & $(B0,C3,B1,C1,B2,C2,B3,C0)$ & $\{A0,B2,B3,C2\}$ \\
		(iii) & $(B0,C0,B1,C1,B2,C2,B3,C3)$ & $\{A0,B2,B3,C2\}$ \\
		(iv) & $(B0,C2,B1,C1,B2,C0,B3,C3)$ & $\{A2,B0,B1,C2\}$ \\
		(v) & $(B0,C3,B1,C0,B2,C1,B3,C2)$ & $\{A3,B1,B2,C0\}$ \\
		(vi) & $(B0,C3,B1,C2,B2,C1,B3,C0)$ & $\{A1,B0,B3,C0\}$ \\
		(vii) & $(B0,C0,B1,C2,B2,C1,B3,C3)$ &  \\
		(viii) & $(B0,C2,B1,C0,B2,C1,B3,C3)$ & $\{A2,B0,B1,C2\}$ \\
		(ix) & $(B0,C3,B1,C1,B3,C0,B2,C2)$ & $\{A1,B3,C0,C1\}$ \\
		(x) & $(B0,C3,B1,C1,B3,C2,B2,C0)$ & $\{A0,B2,B3,C2\}$ \\
		(xi) & $(B0,C3,B1,C0,B3,C1,B2,C2)$ & $\{A1,B3,C0,C1\}$ \\
		(xii) & $(B0,C3,B1,C2,B3,C1,B2,C0)$ &  \\
		(xiii) & $(B0,C0,B1,C3,B3,C1,B2,C2)$ & $\{A3,B1,C0,C3\}$ \\
		(xiv) & $(B0,C2,B1,C3,B3,C1,B2,C0)$ & $\{A2,B0,B1,C2\}$ \\
		(xv) & $(B0,C0,B2,C1,B1,C3,B3,C2)$ & $\{A0,B3,C2,C3\}$ \\
		(xvi) & $(B0,C2,B2,C1,B1,C3,B3,C0)$ & $\{A1,B0,B3,C0\}$ \\
		(xvii) & $(B0,C0,B2,C1,B1,C2,B3,C3)$ & $\{A0,B3,C2,C3\}$ \\
		(xviii) & $(B0,C2,B2,C1,B1,C0,B3,C3)$ &  \\
		(xix) & $(B0,C0,B2,C2,B1,C1,B3,C3)$ & $\{A2,B1,C1,C2\}$ \\
		(xx) & $(B0,C2,B2,C0,B1,C1,B3,C3)$ & $\{A3,B1,B2,C0\}$ \\
	\end{tabular}
	\caption{Possible cycles for $n=3,v=4$: Case (II)}
	\label{tab:n3v4cycles}
\end{table}

We come to the case $n=4$ and $v=4$, the original problem. An exhaustive consideration of the many cases along the lines of the proof of Theorem~\ref{thm:n3v4} is infeasible. However, by computer search we are able to enumerate the non-isomorphic graphs; there are in fact six of them. In Table~\ref{tab:scp44} we show these graphs by specifying the 8-cycles of edges between cliques in the form above. Graph~1 in the table is $\Gamma(4,4)$.

For $v=4$ and $n\geq 5$, the enumeration becomes increasingly challenging. At $n=5$ the computer search yields exactly 24 graphs; at $n=6$ there are at least 129, and at $n=7$ at least 828. 

\begin{table}[h]\centering
	\begin{tabular}{|cll|}
		\hline
		Graph & Connections between cliques & Automorphism group \\
			\hline
			1 & $(A0,B2,A3,B1,A2,B0,A1,B3)$ & $D_{32}$ \\
			& $(A0,C2,A3,C1,A2,C0,A1,C3)$ & \\
			& $(B0,C2,B3,C1,B2,C0,B1,C3)$ & \\
			& $(A0,D2,A3,D1,A2,D0,A1,D3)$ & \\
			& $(B0,D2,B3,D1,B2,D0,B1,D3)$ & \\
			& $(C0,D2,C3,D1,C2,D0,C1,D3)$ & \\
		  \hline
			2 & $(A0,B2,A3,B1,A2,B0,A1,B3)$ & $\Z_2\times\Z_2$ \\
		    & $(A0,C2,A3,C1,A2,C0,A1,C3)$ & \\
		    & $(B0,C2,B3,C1,B2,C0,B1,C3)$ & \\
		    & $(A0,D2,A3,D1,A2,D0,A1,D3)$ & \\
		    & $(B0,D2,B3,D1,B2,D0,B1,D3)$ & \\
		    & $(C0,D1,C2,D3,C1,D0,C3,D2)$ & \\
		  \hline
			3 & $(A0,B2,A3,B1,A2,B0,A1,B3)$ & $\Z_2$ \\
			& $(A0,C2,A3,C1,A2,C0,A1,C3)$ & \\
			& $(B0,C2,B3,C1,B2,C0,B1,C3)$ & \\
			& $(A0,D2,A3,D1,A2,D0,A1,D3)$ & \\
			& $(B0,D1,B2,D3,B1,D0,B3,D2)$ & \\
			& $(C0,D1,C2,D3,C1,D0,C3,D2)$ & \\
		  \hline
			4 & $(A0,B2,A3,B1,A2,B0,A1,B3)$ & $\Z_2\times\Z_2$ \\
			& $(A0,C2,A3,C1,A2,C0,A1,C3)$ & \\
			& $(B0,C2,B3,C1,B2,C0,B1,C3)$ & \\
			& $(A0,D0,A1,D3,A3,D2,A2,D1)$ & \\
			& $(B0,D0,B1,D3,B3,D2,B2,D1)$ & \\
			& $(C0,D0,C1,D3,C3,D2,C2,D1)$ & \\
		  \hline
			5 & $(A0,B2,A3,B1,A2,B0,A1,B3)$ & $\Z_2$ \\
			& $(A0,C2,A3,C1,A2,C0,A1,C3)$ & \\
			& $(B0,C1,B2,C3,B1,C0,B3,C2)$ & \\
			& $(A0,D2,A3,D1,A2,D0,A1,D3)$ & \\
			& $(B0,D1,B2,D3,B1,D0,B3,D2)$ & \\
			& $(C0,D1,C2,D0,C3,D2,C1,D3)$ & \\
		  \hline
			6 & $(A0,B2,A3,B1,A2,B0,A1,B3)$ & $\Z_2\times\Z_2$ \\
			& $(A0,C2,A3,C1,A2,C0,A1,C3)$ & \\
			& $(B0,C1,B2,C3,B1,C0,B3,C2)$ & \\
			& $(A0,D0,A1,D3,A3,D2,A2,D1)$ & \\
			& $(B0,D0,B1,D1,B2,D3,B3,D2)$ & \\
			& $(C0,D0,C1,D1,C2,D3,C3,D2)$ & \\
			\hline
	\end{tabular}
	\caption{The maximal strongly $(4,4)$-clique-partitioned graphs}
	\label{tab:scp44}
\end{table}

\section{Weakly clique-partitioned graphs}\label{sec:weak}
To answer Question 2, we again derive a simple upper bound based on counting edges. Recall that in this version of the problem, we allow the possibility of additional $v$-cliques in our graph but require that the decomposition into $n$ vertex-disjoint $v$-cliques be unique.

\begin{lemma}\label{lem:p2}
Let $n\geq 2$ and $v\geq 2$. Let $G$ be a weakly $(n,v)$-clique-partitioned graph. Then the number of edges in $G$ is at most
\[\binom{nv}{2}-\frac{n(n-1)v}{2}.\]
\end{lemma}

\begin{proof}
To get an upper bound on the number of edges in $G$, we find a lower bound on the number of edges we must remove from a complete graph to make the weakly clique-partitioned property hold.

We know $G$ has $nv$ vertices and can be partitioned into $n$ vertex-disjoint $v$-cliques. Let $N$ be the number of edges removed from the complete graph $K_{nv}$ to obtain $G$. 

Since there are $n(n-1)/2$ pairs of $v$-cliques in the decomposition of $G$, if $N<n(n-1)v/2$ then there must exist some pair $X,Y$ of $v$-cliques in the decomposition which have fewer than $v$ edges missing between them. So there is some $x\in X$ connected to every vertex of $Y$, and some $y\in Y$ connected to every vertex of $X$. 

This leads to a $v$-clique decomposition of $G$ including new cliques $(X\setminus\{x\})\cup\{y\}$ and $(Y\setminus\{y\})\cup\{x\}$, contrary to the uniqueness of the decomposition of $G$. So $N\geq n(n-1)v/2$ and so an upper bound for the number of edges in $G$ is
\[\binom{nv}{2}-\frac{n(n-1)v}{2}.\]
\end{proof}

Again we will call a weakly clique-partitioned graph attaining this bound \emph{maximal}. To show that maximal graphs exist, we define a new graph $\Gamma'(n,v)$ with the same vertex set as before: $\{(i,j):0\leq i\leq n-1,0\leq j\leq v-1\}$.

To define the edges of $\Gamma'(n,v)$, we begin with the complete graph on this vertex set and remove all edges joining $(i,0)$ to $(k,\ell)$ for all $i=0,1,\ldots,n-2$, all $k=i+1,\ldots,n-1$ and all $\ell=0,1,\ldots,v-1$.

The number of edges removed from $\Gamma'(n,v)$ is $(n-1)v+(n-2)v+\cdots+v=n(n-1)v/2$ and so the graph attains the bound of Lemma~\ref{lem:p2}.

Notice that any two vertices with the same first coordinate remain adjacent in $\Gamma'(n,v)$, so each set of $v$ vertices of the form $(i,j)$ for fixed $i$ forms a $v$-clique, which we number $i$. 

This construction is illustrated in Figure~\ref{fig:gammaprime} for the case $n=4,v=4$. In Figure~\ref{fig:gammaprime}(a) we remove all edges from vertex $(0,0)$ to vertices in cliques $1,2,3$. In Figure~\ref{fig:gammaprime}(b) we remove all edges from vertex $(1,0)$ to vertices in cliques $2,3$ and in Figure~\ref{fig:gammaprime}(c) we remove all edges from vertex $(2,0)$ to vertices in clique $3$.

\begin{figure}\centering
	\begin{tabular}{ccc}
		\begin{tikzpicture}[x=0.2mm,y=-0.2mm,inner sep=0.2mm,scale=0.6,thin,vertex/.style={circle,draw,minimum size=9,fill=white}]
\tiny
\node at (150,150) [vertex] (v1) {$00$};
\node at (210,150) [vertex] (v2) {$01$};
\node at (210,210) [vertex] (v3) {$02$};
\node at (150,210) [vertex] (v4) {$03$};
\node at (350,190) [vertex] (v5) {$10$};
\node at (410,190) [vertex] (v6) {$11$};
\node at (410,250) [vertex] (v7) {$12$};
\node at (350,250) [vertex] (v8) {$13$};
\node at (310,390) [vertex] (v9) {$20$};
\node at (370,390) [vertex] (v10) {$21$};
\node at (370,450) [vertex] (v11) {$22$};
\node at (310,450) [vertex] (v12) {$23$};
\node at (110,350) [vertex] (v13) {$30$};
\node at (170,350) [vertex] (v14) {$31$};
\node at (170,410) [vertex] (v15) {$32$};
\node at (110,410) [vertex] (v16) {$33$};
\path
	(v1) edge (v5)
	(v1) edge (v6)
	(v1) edge (v7)
	(v1) edge (v8)
	(v1) edge (v9)
	(v1) edge (v10)
	(v1) edge (v11)
	(v1) edge (v12)
	(v1) edge (v13)
	(v1) edge (v14)
	(v1) edge (v15)
	(v1) edge (v16)
	;
\end{tikzpicture}~~ & 
		\begin{tikzpicture}[x=0.2mm,y=-0.2mm,inner sep=0.2mm,scale=0.6,thin,vertex/.style={circle,draw,minimum size=9,fill=white}]
\tiny
\node at (150,150) [vertex] (v1) {$00$};
\node at (210,150) [vertex] (v2) {$01$};
\node at (210,210) [vertex] (v3) {$02$};
\node at (150,210) [vertex] (v4) {$03$};
\node at (350,190) [vertex] (v5) {$10$};
\node at (410,190) [vertex] (v6) {$11$};
\node at (410,250) [vertex] (v7) {$12$};
\node at (350,250) [vertex] (v8) {$13$};
\node at (310,390) [vertex] (v9) {$20$};
\node at (370,390) [vertex] (v10) {$21$};
\node at (370,450) [vertex] (v11) {$22$};
\node at (310,450) [vertex] (v12) {$23$};
\node at (110,350) [vertex] (v13) {$30$};
\node at (170,350) [vertex] (v14) {$31$};
\node at (170,410) [vertex] (v15) {$32$};
\node at (110,410) [vertex] (v16) {$33$};
\path
	(v5) edge (v9)
	(v5) edge (v10)
	(v5) edge (v11)
	(v5) edge (v12)
	(v5) edge (v13)
	(v5) edge (v14)
	(v5) edge (v15)
	(v5) edge (v16)
	;
\end{tikzpicture}~~ & 
		\begin{tikzpicture}[x=0.2mm,y=-0.2mm,inner sep=0.2mm,scale=0.6,thin,vertex/.style={circle,draw,minimum size=9,fill=white}]
\tiny
\node at (150,150) [vertex] (v1) {$00$};
\node at (210,150) [vertex] (v2) {$01$};
\node at (210,210) [vertex] (v3) {$02$};
\node at (150,210) [vertex] (v4) {$03$};
\node at (350,190) [vertex] (v5) {$10$};
\node at (410,190) [vertex] (v6) {$11$};
\node at (410,250) [vertex] (v7) {$12$};
\node at (350,250) [vertex] (v8) {$13$};
\node at (310,390) [vertex] (v9) {$20$};
\node at (370,390) [vertex] (v10) {$21$};
\node at (370,450) [vertex] (v11) {$22$};
\node at (310,450) [vertex] (v12) {$23$};
\node at (110,350) [vertex] (v13) {$30$};
\node at (170,350) [vertex] (v14) {$31$};
\node at (170,410) [vertex] (v15) {$32$};
\node at (110,410) [vertex] (v16) {$33$};
\path
	(v9) edge (v13)
	(v9) edge (v14)
	(v9) edge (v15)
	(v9) edge (v16)
	;
\end{tikzpicture}~~\\
		(a) & (b) & (c)
	\end{tabular}

	\caption{Construction of the graph $\Gamma'(4,4)$ by deleting edges}
	\label{fig:gammaprime}
\end{figure}

\begin{theorem}\label{thm:p2}
Let $n\geq 2$ and $v\geq 2$. Then $\Gamma'(n,v)$ is weakly $(n,v)$-clique-partitioned and hence the upper bound of Lemma~\ref{lem:p2} is attained.
\end{theorem}

\begin{proof}
We proceed iteratively. Vertex $(0,0)$ is not adjacent to any vertex not in clique number 0, so any $v$-clique decomposition must include clique 0.

In the remainder of the graph, vertex $(1,0)$ is not adjacent to any vertex not in clique 1, so any $v$-clique decomposition must include clique 1 also. 

This argument can be repeated for each numbered $v$-clique, and so there is a unique decomposition of $\Gamma'(n,v)$ into $n$ $v$-cliques as required.
\end{proof}

In contrast with the strongly clique-partitioned case in Section~\ref{sec:strong}, it turns out that these graphs are the unique weakly clique-partitioned graphs attaining the upper bound. In the case $v=2$, weakly clique-partitioned graphs are precisely those which admit a unique perfect matching. The structure of edge-maximal graphs in this class was deduced by Hetyei~\cite{Hetyei1964}; see also~\cite[Corollary 1.6]{Lovasz1972}. By Hetyei's results, the graphs $\Gamma'(n,2)$ are the unique maximal weakly clique-partitioned graphs.

Our aim now is to extend this uniqueness result to all $v\geq 2$ and all $n\geq 2$. We begin with the case $n=2$.

\begin{theorem}\label{thm:wcpn2}
Let $v\geq 2$. Let $G$ be a maximal weakly $(2,v)$-clique-partitioned graph. Then $G$ is isomorphic to $\Gamma'(2,v)$.
\end{theorem}

\begin{proof}
$G$ consists of two $v$-cliques $X$ and $Y$ with $v(v-1)$ edges between them; in other words it is the complete graph $K_{2v}$ with $v$ edges missing between the two cliques.

We show first that in one of the two cliques, each vertex is adjacent to precisely $v-1$ vertices in the other clique. For if not, then there are vertices $x\in X$ and $y\in Y$ each adjacent to all vertices in the other clique; then $(X\setminus\{x\})\cup\{y\}$ and its complement form a new $v$-clique partition of $G$. Without loss of generality, we may assume all vertices of $Y$ are adjacent to precisely $v-1$ vertices of $X$. We now show that each vertex of $Y$ must be adjacent to the \emph{same} $v-1$ vertices of $X$.

Let $x\in X$ be a vertex not adjacent to all vertices in $Y$. Let $M$ be the set of non-neighbours of $x$ in $Y$, and let $m=|M|$. Since only $v$ edges are missing between $X$ and $Y$, there are at least $m-1$ vertices in $X$ adjacent to all vertices in $Y$. Let $S$ be a set of $m-1$ such vertices. Then $(Y\setminus M)\cup\{x\}\cup S$ and its complement are a $v$-clique decomposition of $G$. This decomposition is distinct from $X$ and $Y$ unless $m=v$.

Thus there is one vertex in $X$ not adjacent to any vertex in $Y$, and this accounts for all $v$ missing edges between $X$ and $Y$. Thus $G$ is isomorphic to $\Gamma'(2,v)$.
\end{proof}

It is immediate that if $G$ is a weakly $(n,v)$-clique-partitioned graph, then any $m$ of the $n$ $v$-cliques in the unique decomposition of $G$ ($1\leq m\leq n$) induce a weakly $(m,v)$-clique-partitioned subgraph of $G$. The above discussion then shows that if $X$ and $Y$ are two of the $v$-cliques in a maximal graph, then exactly one of the following situations must occur.
\begin{enumerate}[label=(\alph*),topsep=-1ex,itemsep=-1ex]
	\item There is a vertex $x\in X$ not adjacent to any vertex in $Y$;\newline
	or
	\item There is a vertex $y\in Y$ not adjacent to any vertex in $X$.
\end{enumerate}
In the first situation we write $X\to Y$; otherwise $Y\to X$. Where necessary, we will indicate the distinguished vertex not adjacent to the other clique by the notation $X^{(x)}\to Y$ or $Y^{(y)}\to X$ as appropriate. Our first result towards the extension to $n\geq 3$ shows that in fact these distinguished vertices must be unique within a clique.

\begin{lemma}\label{lem:distinguished}
Let $n\geq 3$ and let $v\geq 2$. Let $G$ be a maximal weakly $(n,v)$-clique-partitioned graph, and let $X=\{x_1,\ldots,x_v\}$, $Y=\{y_1,\ldots,y_v\}$ and $Z=\{z_1,\ldots,z_v\}$ be distinct cliques in the decomposition of $G$. If $X^{(x_i)}\to Y$ and $X^{(x_j)}\to Z$ then $i=j$.
\end{lemma}

\begin{proof}
The subgraph of $S$ of $G$ induced by $X\cup Y\cup Z$ is weakly $(3,v)$-clique partitioned. Suppose $i\neq j$. Up to isomorphism, we may assume $i=1$, $j=2$ and $Y^{(y_1)}\to Z$. Then the sets $\{x_1,z_2,\ldots,z_v\}$, $\{z_1,y_2,\ldots,y_v\}$ and $\{y_1,x_2,\ldots,x_v\}$ form another $v$-clique partition of $S$, a contradiction.
\end{proof}

In light of Lemma~\ref{lem:distinguished} we may drop the superscript notation and assume the vertices are numbered such that if $X\to Y$, then the distinguished vertex in $X$ is $x_1$. 

If $X$ and $Y$ are $v$-cliques in the decomposition of $G$, then either $X\to Y$ or $Y\to X$. The cliques are therefore arranged in a tournament. The next step is to show that this tournament is acyclic.

\begin{lemma}\label{lem:acyclic}
Let $n\geq 3$ and let $v\geq 2$. Let $G$ be a maximal weakly $(n,v)$-clique-partitioned graph, and let $X=\{x_1,\ldots,x_v\}$, $Y=\{y_1,\ldots,y_v\}$ and $Z=\{z_1,\ldots,z_v\}$ be distinct cliques in the decomposition of $v$. If $X\to Y$ and $Y\to Z$ then $X\to Z$.
\end{lemma}

\begin{proof}
If $Z\to X$, then the sets $\{x_1,z_2,\ldots,z_v\}$, $\{y_1,x_2,\ldots,x_v\}$ and $\{z_1,y_2,\ldots,y_v\}$ form a $v$-clique decomposition of $G$.
\end{proof}

We are now ready to state the uniqueness result.

\begin{theorem}
Let $n\geq 2$ and let $v\geq 2$. Let $G$ be a maximal weakly $(n,v)$-clique-partitioned graph. Then $G$ is isomorphic to the graph $\Gamma'(n,v)$.
\end{theorem}

\begin{proof}
The result follows from Theorem~\ref{thm:wcpn2}, Lemmas~\ref{lem:distinguished} and~\ref{lem:acyclic} and uniqueness of the acyclic tournament (Lemma~\ref{lem:tournament}).
\end{proof}

The graphs $\Gamma'(n,v)$ are of course far from regular. However, their structure is sufficiently tightly defined to be able to compute their automorphism group.

\begin{proposition}
Let $n\geq 2$ and let $v\geq 2$. Then $\Aut(\Gamma'(n,v))$ has order $v!\left((v-1)!\right)^{n-1}$ and is isomorphic to $S_v\times S_{v-1}^{n-1}$.
\end{proposition}

\begin{proof}
If $X_1,X_2,\ldots,X_n$ are the $n$ $v$-cliques in the composition of $\Gamma'(n,v)$, then by a suitable labelling they form a chain $X_1\to X_2\to\cdots\to X_n$. The distinguished vertices in $X_i$, $1\leq i\leq n-1$ have valency $i(v-1)$. Non-distinguished vertices in $X_i$, $1\leq i\leq n-1$ have valency $i(v-1)+(n-i)v=nv-i$. Any automorphism must clearly preserve the unique clique decomposition, but no cliques can be exchanged since the distinguished vertices all have different valency. The distinguished vertices are fixed by any automorphism, but the non-distinguished vertices within any clique may be permuted freely. Clique $X_n$ has no distinguished vertex. The result follows.
\end{proof}

\section{Further links to extremal graph theory}\label{sec:links}
Recall the following which is well-known. An $n$-vertex graph which does not contain any $(r+1)$-vertex clique may be constructed by partitioning the set of vertices into $r$ parts of equal or nearly equal size, and connecting two vertices by an edge whenever they belong to different parts. This is the Tur\'an graph $T(n,r)$. By ``nearly equal size'' is meant that the cardinality of any two parts differs by at most one. This graph has the largest number of edges among all $K_{r+1}$-free $n$-vertex graphs~\cite{Turan1941,Turan1954}. Putting $r=v$ and replacing $n$ by $nv$, the graph $T(nv,v)$ is regular of valency $n(v-1)$ and has $n^2 v(v-1)/2$ edges. This is precisely the same number of edges as $\Gamma'(n,v)$, the unique maximal weakly $(n,v)$-clique-partitioned graph on this parameter set. As observed in the previous section, the graphs $\Gamma'(n,v)$ are not regular and arise in a different manner to the graphs $T(nv,v)$. But it is not without some interest that there exist two distinct families of unique graphs on the same number of vertices with completely different properties.

Probably of more interest though in this context are maximal strongly $(n,v)$-clique-partitioned graphs. Clearly these graphs contain no $(v+1)$-clique. Let $G$ be a maximal strongly $(n,v)$-clique-partitioned graph and denote the number of edges by $|G|$. Then $|G|=nv(v-1)/2+nv(n-1)(v-2)/2=nv(nv-2n+1)/2$. So $|G|/|T(nv,v)|=(nv-2n+1)/n(v-1)=1-(n-1)/n(v-1)\to 1$ as $v\to\infty$. Thus maximal strongly $(n,v)$-clique-partitioned graphs form a family of graphs avoiding $(v+1)$-cliques, the number of edges of which approach asymptotically the number of edges in the extremal graphs $T(nv,v)$.

Finally, we return to the relationship between our questions and colouring problems. It is shown in \cite[Theorem~3.3]{Klotz2016} that the circulant graph $\Cay(\Z_{nv},\{\pm 1,\pm 2,\ldots,\pm(n-1)\})$ is an edge-minimal uniquely $n$-colourable graph of order $nv$. Since by \cite[Theorem~2.2]{Klotz2016} the colour classes of such a Cayley graph must be cosets of some subgroup of order $v$ in $\Z_{nv}$, it follows that this circulant graph is unique. Therefore its complement is the unique maximal strongly $(n,v)$-clique partitioned circulant graph. This is in agreement with our Proposition~\ref{prop:circulant}. 

\section*{Acknowledgements}
The third author acknowledges support from the APVV Research Grants 17-0428 and 19-0308, and the VEGA Research Grants 1/0206/20 and 1/0567/22.

We thank the anonymous referees for their helpful comments, especially in relation to the links between our work and complementary work on colourings.


\end{document}